\documentclass[english, 11pt]{amsart}

\usepackage{tikz}
\usetikzlibrary{cd}
\usepackage{aliascnt}
\usepackage{mathrsfs}
\usepackage[all,poly,knot]{xy}
\usepackage{hyperref}
\usepackage{comment}  
\usepackage{csquotes}  
\usepackage{amssymb,amsbsy,amsmath,amsfonts,amssymb,amscd,
	graphics,color,footmisc,fancyhdr,multicol,fancybox,
	graphicx,mathrsfs,rotating,ifthen,wasysym}

\usepackage{times}
\usepackage{pdfpages}
\usepackage{multirow}
\usepackage{nccrules}
\usepackage{textcomp}
\usepackage{comment}
\usepackage{framed}
\usepackage[all]{xy}
\usepackage{graphicx}
\usepackage{pgf,tikz}
\usepackage{mathrsfs}
\usetikzlibrary{arrows}
\usepackage{lipsum}
\usepackage{mathtools}

\DeclareMathOperator{\pr}{pr}

\def\log{\mathrm{log}\,}

\theoremstyle{plain}

\newtheorem{thm}{Theorem}[section]

\newtheorem{lem}[thm]{{Lemma}}

\newtheorem{pro}[thm]{Proposition}

\newtheorem{conj}[thm]{{Conjecture}}

\newtheorem{ques}[thm]{{Question}}

\newtheorem{defi}[thm]{Definition}

\theoremstyle{remark}
\newtheorem{rmk}[thm]{Remark}

\numberwithin{equation}{section}

\usepackage[top=1.5in, bottom=1.5in, left=1in, right=1in]{geometry}

 \usepackage[hyperpageref]{backref}
 
\usepackage{xcolor}
\hypersetup{
	colorlinks,
	linkcolor={red!50!black},
	citecolor={blue!62!black},
	urlcolor={blue!80!black}
}
\theoremstyle{plain}
\newcommand{\thistheoremname}{}
\newtheorem*{genericthm*}{\thistheoremname}
\newenvironment{namedthm*}[1]{\renewcommand{\thistheoremname}{#1}%
	\begin{genericthm*}}
	{\end{genericthm*}}

\newtheoremstyle{named}{}{}{\itshape}{}{\bfseries}{.}{.5em}{\thmnote{#3's }#1}
\theoremstyle{named}

\makeatletter
\newcommand\thankssymb[1]{\textsuperscript{\@fnsymbol{#1}}}
\makeatother

\begin{document} 
	\title[Generating all Ahlfors currents by a single entire curve]{Generating all Ahlfors currents by a single entire curve} 


\author{Yunling Chen}
\address{Academy of Mathematics and Systems Sciences, Chinese Academy of Sciences, Beijing 100190, China}
\email{chenyl25@amss.ac.cn}

\author{John Erik Fornæss}
\address{Department of Mathematics, NTNU, Norway}
\email{fornaess@gmail.com}

\author{Song-Yan Xie}
\address{State Key Laboratory of Mathematical Sciences, Academy of Mathematics and Systems Science, Chinese Academy of Sciences, Beijing 100190, China;  School of Mathematical Sciences, University of Chinese Academy of Sciences, Beijing 100049, China.}
	\email{xiesongyan@amss.ac.cn}
	
	\date{\today}

\begin{abstract}
Let \(X\) be a compact complex manifold possessing the \emph{Runge approximation property on discs}, meaning that every holomorphic map from a  closed disc into \(X\) is approximable by a global holomorphic map from \(\mathbb{C}\). We construct an entire curve \(F : \mathbb{C} \to X\) such that the associated family of concentric holomorphic discs \(\{F|_{\overline{\mathbb{D}}_r}\}_{r>0}\) generates  all Ahlfors currents on \(X\), thereby settling a conjecture of Sibony.
\end{abstract}

\maketitle
%\tableofcontents 
\section{\bf Introduction}

The investigation of holomorphic entire curves in complex manifolds forms a classical theme in complex geometry (cf.~\cite{NW14, Ru21}), originating in Nevanlinna’s value distribution theory \cite{Nevanlinna25}. This analytic perspective was later enriched by Ahlfors’ geometric insights through his covering surface theory \cite{Ahlfors35}.

The advancement of these theories is largely driven by their connection to  complex hyperbolicity  problems. In a pivotal contribution, McQuillan \cite{McQuillan98} verified the famous Green–Griffiths  conjecture for complex projective surfaces with $c_1^2 > c_2$ by means of  \emph{Nevanlinna currents}, which capture the asymptotic behavior of entire curves. Later,  Duval \cite{Duval08} employed a simplified variant of Nevanlinna currents, called \emph{Ahlfors currents}, to establish a deep,  quantitative refinement of Brody's lemma, thereby obtaining a characterization of complex hyperbolicity in terms of an isoperimetric inequality for holomorphic discs. For further 
important applications of Ahlfors and Nevanlinna currents, we refer the reader to \cite{DinhSibony18, DuvalHuynh18, DinhVu20, HuynhVu21, Duj22}, among others.

\smallskip
We begin by recalling the basic definitions and notations. Let \( X \) be a compact complex manifold endowed with a Hermitian (1,1)-form \( \omega \). Throughout this paper, \( \mathbb{D}_r \) and \( \mathbb{D}(c, r) \) denote the discs in \( \mathbb{C} \) of radius \( r \), centered at the origin and at \( c \), respectively. By a holomorphic map \( f \colon \overline{\mathbb{D}}_R \to X \), we mean that \( f \) is holomorphic on some open neighborhood \( U \) of the closed disc \( \overline{\mathbb{D}}_R \).

We associate to \( f \) a positive bidimension \((1,1)\)-current \( A_f \), called the {\em Ahlfors-type current}, which
acts on any smooth \((1,1)\)-form \( \eta \) on \( X \) by normalized integration:
\begin{equation}\label{eq:current def}
    A_f(\eta) = \frac{1}{\mathrm{Area}_\omega(f(\mathbb{D}_R))} \int_{\mathbb{D}_R} f^*\eta.
\end{equation}
 The normalization gives \( A_f \) the mass \( A_f(\omega) = 1 \).

From a sequence of nonconstant holomorphic discs \( \{ f_i \colon \overline{\mathbb{D}}_{r_i} \to X \}_{i \geqslant 1} \) which satisfies the length-area condition
\begin{equation}
    \label{eq:length-area-Ahlfors}
\lim_{i \to \infty} \frac{\mathrm{Length}_\omega(f_i(\partial \mathbb{D}_{r_i}))}{\mathrm{Area}_\omega(f_i(\mathbb{D}_{r_i}))} = 0,
\end{equation}
we obtain an associated sequence of Ahlfors-type currents \( \{ A_{f_i} \}_{i \geqslant 1} \) with bounded mass. By the Banach--Alaoglu theorem, 
some subsequence converges weakly to a positive  \((1,1)\)-current \( \mathcal{T} \), which is 
closed due to the  condition~\eqref{eq:length-area-Ahlfors}. Such a limit \( \mathcal{T} \) is called an \emph{Ahlfors current} on $X$.

By Ahlfors' lemma (cf.~\cite[p.~7]{Duval21}), for any entire curve \( f \colon \mathbb{C} \to X \), there exists a sequence of increasing radii \( \{r_i\}_{i \geqslant 1} \) tending to infinity such that the corresponding concentric holomorphic discs \( \{f|_{\overline{\mathbb{D}}_{r_i}}\}_{i \geqslant 1} \) satisfy the length-area condition~\eqref{eq:length-area-Ahlfors}. Consequently, one can obtain an associated Ahlfors current $\mathcal{T}$ which encodes the asymptotic geometric behavior of $f$. Obviously, Ahlfors currents are natural  analogues of the Lelong currents of integration over analytic subvarieties~\cite{Lelong57}.
 %It is also an interesting fact that the use of Ahlfors currents in complex dynamics predates that of Nevanlinna currents (see, e.g.,~\cite{FornaessSibony94}). 

In the context of complex dynamics, Ahlfors currents associated to certain holomorphic curves have been shown to exhibit uniqueness properties \cite{DinhSibony18, DinhVu20}. These examples motivate the following fundamental question:

\begin{ques}\label{uniqueness-question}\rm
 Are all Ahlfors (resp. \ Nevanlinna) currents associated to the same entire curve cohomologically equivalent?
\end{ques}

Huynh and Xie~\cite{HuynhXie21} first answered the above question in the negative through explicit counterexamples. This fundamental obstruction rules out a naive analytic intersection theory between entire curves and divisors via pairing of the corresponding Ahlfors (resp.\ Nevanlinna) currents with Chern forms.

\medskip
Following his review of an early manuscript~\cite{HuynhXie21}, the late Nessim Sibony (1947--2021) proposed the following striking conjecture: \footnote{\,Sibony (personal communication, 2021) pointed out: ``I suspect that for $\mathbb{P}^n$, or tori, all the Ahlfors currents can be obtained just using one map; see the discussion on the Birkhoff approach~\cite{Birkhoff29} in the enclosed paper~\cite{DinhSibony20}.'' The citations are added by the authors. }
\begin{conj}\label{conj:sibony}
For certain compact complex manifolds $X$, such as $\mathbb{P}^n$ or complex tori, every Ahlfors current on $X$ can be obtained from a single entire curve $f \colon \mathbb{C} \to X$.
\end{conj}

In the same private correspondence, Sibony pointed out a connection between this conjecture and Problem~9.1 in his joint open problem list with Dinh~\cite{DinhSibony20}, concerning  universal entire functions in the sense of Birkhoff~\cite{Birkhoff29}. 
Recently, Dinh-Sibony's Problem 9.1 has spurred a series of recent developments~\cite{ChenHuynhXie23, GuoXie24, GuoXie25}.
%Following Sibony's suggestion, one may consider the natural extension to arbitrary holomorphic discs $\{f|_{\mathbb{D}(a,r)}\}_{a\in\mathbb{C}, r>0}$ rather than restricting to concentric ones $\{f|_{\mathbb{D}_r}\}_{r>0}$ for constructing Nevanlinna/Ahlfors currents associated with $f$.
%Sibony also mentioned~\cite{Hormander89} as a potentially relevant reference, though the precise intended application remains unclear.

\medskip
Inspired by Sibony's hint on Birkhoff’s method~\cite{Birkhoff29}, Xie~\cite{Xie24} introduced the following concept.

\begin{defi}[{\cite[Definition 1.7]{Xie24}}]\label{def:xwoka1}\rm 
A connected complex manifold $X$ equipped with a complete distance function $\mathrm{d}_X$ (whose existence is guaranteed by \cite{NomizuOzeki1961}) is said to have the \emph{weak Oka-1 property} if the following holds: for any two disjoint closed discs $D_1, D_2$ contained in a larger closed disc $D \subset \mathbb{C}$, any holomorphic map $F \colon U \to X$ defined on a neighborhood $U$ of $D_1 \cup D_2$, and any error bound $\mathsf{e} > 0$, there exists a holomorphic map $\hat{F} \colon \hat{U} \to X$ on a neighborhood $\hat{U}$ of $D$ such that
\[
\mathrm{d}_X\big(\hat{F}(z), F(z)\big) 
\leqslant \mathsf{e}, \quad \forall z \in D_1 \cup D_2.
\]
\end{defi}

This class of manifolds includes all Oka manifolds~\cite{Forstneric17} (notably $\mathbb{P}^n$ and complex tori) and embraces the broader class of Oka-1 manifolds (Definition~\ref{def:Oka1}) introduced by Alarc\'on and Forstneri\v{c}~\cite{AlarconForstneric23}, which encompasses examples such as Kummer surfaces and  elliptic $K3$ surfaces. Further examples  can be found in~\cite{BW25} (rationally simply connected projective manifolds) as well as in~\cite{FL25}.

\medskip
For such manifolds, Xie obtained the following partial answer to Sibony's Conjecture~\ref{conj:sibony}:

\begin{thm}[{\cite[Theorem A]{Xie24}}]
\label{thm:Xie24}
Let $X$ be a compact complex manifold satisfying the weak Oka-1 property. Then there exists an entire curve $f \colon \mathbb{C} \to X$ whose associated holomorphic discs $\{f|_{\overline{\mathbb{D}}(a,r)}\}_{a\in\mathbb{C},\, r>0}$ generate all Nevanlinna and Ahlfors currents on $X$.
\end{thm}

The proof employs Birkhoff's method to ``patch'' together infinitely many suitably chosen holomorphic discs $\{f_i \colon \overline{\mathbb{D}}_{\rho_i} \to X\}_{i \geqslant 1}$ (see Lemma~\ref{lem-countable-discs}). However, this strategy fails when one is restricted to using concentric discs from a single entire curve $f \colon \mathbb{C} \to X$. The core challenge is that, for any two radii $r_1 < r_2$, the behavior of $f$ on $\overline{\mathbb{D}}_{r_1}$ inherently influences its behavior on $\overline{\mathbb{D}}_{r_2}$. This ``pollution'' of information makes it unclear how to ``amalgamate'' the collection $\{f_i\}_{i\geqslant 1}$ into a family of concentric discs from the same entire curve.

\smallskip

Nevertheless, in a recent work~\cite{WuXie25}, Wu and Xie were able to verify Sibony's Conjecture~\ref{conj:sibony} for complex tori at a cohomological level, by means of Hodge theory and a  variant of the ``ping-pong'' strategy in~\cite{Xie24}.

\begin{thm}[{\cite[Theorem A]{WuXie25}}]\label{thm:wu-xie}
For any complex torus $\mathbb{T}$, there exists an entire curve $f \colon \mathbb{C} \to \mathbb{T}$ such that the  concentric holomorphic discs
$\{f|_{{\overline{\mathbb{D}}}_R}\}_{R > 0}
$
generate all Nevanlinna and Ahlfors currents on $\mathbb{T}$, up to cohomological equivalence.
\end{thm}

\smallskip
Looking back, Sibony's Conjecture~\ref{conj:sibony} anticipates the remarkable flexibility of entire curves in Oka geometry (cf.~\cite{Forstneric17}). In this work, we establish this conjecture in full generality for weak Oka-1 manifolds.

\smallskip
We begin by providing two alternative characterizations of weak Oka-1 manifolds.

\begin{defi}\label{def:vwoka1}\rm
A  connected complex manifold $X$ with a complete distance function $\mathrm{d}_X$ has the {\em local Runge approximation property on discs} if for  any holomorphic map $f_1 \colon U \to X$ defined on a neighborhood $U$ of $\overline{\mathbb{D}}$, and any $\epsilon > 0$, there exists a holomorphic map ${f}_2 \colon V \to X$ defined on a neighborhood $V$ of $\overline{\mathbb{D}}_2$ satisfying $\sup_{z \in \overline{\mathbb{D}}} \mathrm{d}_X({f}_1(z), f_2(z)) 
\leqslant\epsilon$.
\end{defi}

\begin{defi}\label{def:awoka1}\rm
A  connected complex manifold $X$  
is said to have the {\em Runge approximation property on discs}, if any holomorphic map from a neighborhood of a closed disc $D\subset \mathbb{C}$ to $X$ can be approximated uniformly on $D$ by entire curves from $\mathbb{C}$ to $X$.
\end{defi}

The above two properties are actually equivalent to the weak Oka-1 property (Definition~\ref{def:xwoka1}), as established in Proposition~\ref{equivalence of definitions}. This class of complex manifolds can be regarded as the antithesis of complex hyperbolic manifolds.

\begin{thm}[Main Theorem]\label{thm:main}
For every compact complex manifold $X$ satisfying the Runge approximation property on discs, there exists an entire curve $F \colon \mathbb{C} \to X$ such that the  concentric holomorphic discs $\{F|_{\overline{\mathbb{D}}_r}\}_{r>0}$ generate all Ahlfors currents on $X$.
\end{thm}

This result reveals a remarkable flexibility of entire curves in Oka geometry. By showing that the concentric discs from a single entire curve are sufficient, it strengthens Theorem~\ref{thm:Xie24} and thereby fully resolves Sibony's conjecture~\ref{conj:sibony} for weak Oka-1 manifolds.

\smallskip
The proof of Theorem~\ref{thm:main} is based on a delicate induction process, synthesizing ideas from three key sources:
\begin{itemize}
\item[(i)] Birkhoff's method of constructing universal entire functions from countable families of holomorphic functions on discs~\cite{Birkhoff29};

\smallskip
\item[(ii)] the gluing technique in several complex variables, pioneered by Fornaess and Stout~\cite{FornaessStout77a, FornaessStout77b};

\smallskip
\item[(iii)] the ``ping-pong'' strategy introduced in~\cite{Xie24}.
\end{itemize}

\smallskip
Although the technical machinery involved is sophisticated, the essence of our approach  ultimately reduces to the simple identities \eqref{naive identity 1} and \eqref{naive identity 2}.

Our novel inductive construction of such entire curves can  be also applied to deduce all previous results, established in~\cite{HuynhXie21, Xie24, WuXie25}, on Siu's decompositions of Ahlfors currents.

%Moreover, standard gluing techniques for $J$-holomorphic curves (cf.~\cite{MS12}) allow our method to be extended to the analogous case of almost complex Oka-1 manifolds, which are readily defined. However, to maintain the focus and brevity of the present paper, we will not pursue this direction further.

\medskip
This paper is structured as follows. Section~\ref{sect:A Gluing Technique} recalls  the classical holomorphic approximation theory and explains its application in our setting. Section~\ref{section:2} establishes the equivalence of three characterizations of weak Oka-1 manifolds. The proof of the main theorem is presented in Section~\ref{section:Ahlfors Currents}. Finally, Section~\ref{section:4} discusses three open problems concerning (weak) Oka-1 manifolds.

\section{\bf A Holomorphic Gluing Technique}
\label{sect:A Gluing Technique}

\subsection{ Holomorphic Approximation Theory}
We begin by recalling several key results from holomorphic approximation theory, which will play essential roles in our constructions. For a comprehensive treatment, we refer to \cite{FornaessForstnericWold20} and \cite{AlarconForstneric21}.

Our starting point is the classical theory of holomorphic approximation on the complex plane $\mathbb{C}$. One of the earliest and most influential results in this area is Runge's approximation theorem \cite{Run85}:

\begin{thm} \label{thm:runge approximation}
Let $K$ be a compact subset of $\mathbb{C}$ such that its complement $\mathbb{C}\setminus K$ is path connected. Then any holomorphic function $f$ defined on a neighborhood of $K$ can be approximated uniformly on $K$ by a sequence of  polynomials.
\end{thm}

The theorem of Mergelyan \cite{Mer52} is a profound generalization of Runge's theorem, providing precise conditions on a compact set $K \subset \mathbb{C}$ to ensure that every function in $\mathcal{A}(K)$ (i.e., continuous on $K$ and holomorphic in its interior) can be uniformly approximated by holomorphic functions on neighborhoods of $K$ (cf.~\cite[Section 7]{FornaessForstnericWold20}). This result has been extensively generalized, notably to the context of $\mathcal{C}^r$-approximation on \textit{admissible sets}.

%The theorem of S. N. Mergelyan is another major advancement in holomorphic approximation theory. As a generalization of Runge's theorem, Mergelyan's theorem delicately refines the conditions on the function $f$ to be approximated, requiring that $f$ be continuous on the compact set and holomorphic in its interior. Before tracing the development of Mergelyan's theorem, we first introduce some unified notations.

\begin{defi}[cf.~{\cite[Definition 3]{FornaessForstnericWold20}}] \label{def:admissible-set}\rm
A compact set \( N \) in a Riemann surface \( M \) is called \emph{admissible} if it can be written as \( N = K \cup E \), where:
\begin{itemize}
    \item \( K \) is a finite union of pairwise disjoint compact domains with piecewise \( \mathcal{C}^1 \) boundaries.
    \item \( E \) is a finite union of pairwise disjoint smooth Jordan arcs and closed Jordan curves.
    \item The set \( E \) intersects \( K \) at most at the endpoints of these arcs, and all such intersections are transverse to the boundary \( \partial K \).
\end{itemize}
\end{defi}

%\begin{thm}[Mergelyan’s theorem, cf. {\cite[Theorem 16]{FornaessForstnericWold20}}]\label{Mergelyan’s theorem}
%If \( S \) is an admissible set in a Riemann surface \( M \) without holes, then every function \( f \in \mathcal{A}^r(S) \) (where \( r \in \mathbb{Z}_+ \)) can be approximated in the \( \mathcal{C}^r(S) \) topology by holomorphic functions on \( M \). \qed
%\end{thm}
%A proof of Theorem \ref{Mergelyan’s theorem} by induction on \( r \) is given in , reducing it to the basic case \( r = 0 \) closely related to Mergelyan's theorem in the plane.They also give an analogous result for manifold-valued maps.
\begin{thm}[Mergelyan theorem for manifold-valued maps on admissible sets; cf. {\cite[p.~74]{AlarconForstneric21}}]
\label{thm:mergelyan-manifold-valued}
Let \( M \) be a Riemann surface, \( X \) be a complex manifold, and \( N \) be an admissible set in \( M \). For any \( r \in \mathbb{Z}_{\geqslant 0} \), every map \( f \in \mathcal{A}^r(N, X) \) can be approximated in the \( \mathcal{C}^r(N, X) \) topology by holomorphic maps \( U \to X \) defined on some open neighbourhoods \( U \) of \( N \). 
\end{thm}

Following \cite{FornaessForstnericWold20}, for \( r \in \mathbb{Z}_{\geq 0} \), we denote by \( \mathcal{C}^r(N, X) \) the space of \( r \)-times continuously differentiable maps from a neighborhood of \( N \) into \( X \), and by \( \mathcal{A}^r(N, X) \) the subspace of maps that are holomorphic on the interior \( \mathring{N} \). %The statement of Theorem~\ref{thm:mergelyan-manifold-valued} is presented in~{\cite[p.~74]{AlarconForstneric21}}, which refers the proof to~{\cite[Section~7.2]{FornaessForstnericWold20}}. Since the latter treats a broader context and does not explicitly isolate the proof of our specific formulation, we include in Appendix~B a structured outline of the argument for the reader’s convenience, while omitting the proofs of several auxiliary results that are invoked in the original sources.

\begin{rmk}
\label{how to apply Thm 2.3}
To apply the above theorem to an \( r \)-times continuously differentiable map \( f \colon N \to X \) that is holomorphic on \( \mathring{N} \), we need to extend \( f \) to a neighborhood of \( N \). It is a consequence of Whitney's extension theorem \cite{Whitney34, FornaessForstnericWold20}; see Proposition~\ref{prop:extension-admissible}. 
\end{rmk}

%Throughout this paper, $X$ will denote a compact complex manifold satisfying the weak Oka-1 property. We fix a Hermitian form $\omega$ on $X$, which induces a complete distance function $d_X$ on the manifold.
 %Instrumental to our proof is a gluing lemma, whose underlying technique originates in the works of Fornaess and Stout (\cite[Lemma II.2]{FornaessStout77a}, \cite[Lemma 3]{FornaessStout77b}).

\subsection{Gluing Two Holomorphic Discs}

The following Propositions~\ref{almost geodesic link} and \ref{prop:extension-admissible} are standard facts in differential topology; their proofs are therefore omitted.

\begin{pro}
\label{almost geodesic link}
Let $(M,g)$ be a complete Riemannian manifold. For any two distinct points $p,q \in M$, any tangent vectors $v_p \in T_pM$, $v_q \in T_qM$, and any $\epsilon > 0$, there exists a smooth curve $\gamma: [0,1] \to M$ connecting $p$ to $q$ with $\gamma'(0) = v_p$, $\gamma'(1) = v_q$, and length $L(\gamma) < \mathrm{d}_M(p,q) + \epsilon$.
\qed
\end{pro}

This proposition will be used in \textbf{Trick 3} below.

\begin{pro}
\label{prop:extension-admissible}
Let \( M \) be a Riemann surface and \( N \subset M \) an admissible set (see Definition~\ref{def:admissible-set}). 
Let \( Y \) be a smooth manifold and \( f \colon N \to Y \) a continuously differentiable map. 
Then there exist an open neighborhood \( V \) of \( N \) in \( M \) and a continuously differentiable map \( F \colon V \to Y \) such that \( F|_N = f \).\qed
\end{pro}

In our subsequent application, we take \( M = \mathbb{R}^2 \) and consider an admissible set \( N \subset M \) consisting of two disjoint closed discs (with $\mathcal{C}^{\infty}$ boundaries) connected transversally by a short line segment (see Figure~\ref{fig:dumbbell domain}). %To apply Theorem~\ref{thm:mergelyan-manifold-valued}, we need to extend certain continuously differentiable maps defined on \( N \) to slightly larger open neighborhoods (see Remark~\ref{how to apply Thm 2.3}). 

\begin{figure}[!htbp]
    \centering
    \includegraphics[scale=0.6]{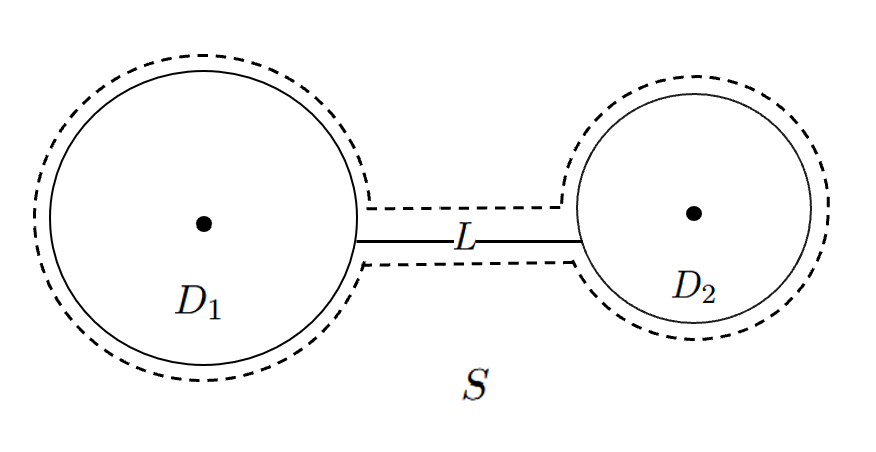}
    \caption{Dumbbell-shaped neighborhood of $D_1$ and $D_2$.}
    \label{fig:dumbbell domain}
\end{figure}

\begin{lem}\label{lem:dumbbell domain holomorphic map}
Let \( X \) be a complex manifold, and let \( D_1, D_2 \subset \mathbb{C} \) be two disjoint closed disks. For \( i = 1, 2 \), let \( f_i \colon U_i \to X \) be a holomorphic map defined on an open neighborhood \( U_i \) of \( D_i \). Then, for any \( \epsilon > 0 \), there exist a dumbbell-shaped domain \( S \subset \mathbb{C} \) containing \( D_1 \cup D_2 \) (see Figure~\ref{fig:dumbbell domain}) and a holomorphic map \( f \colon S \to X \) such that
\[
\| f - f_i \|_{\mathcal{C}^1} < \epsilon \quad \text{on $D_i$ \quad for } i = 1, 2.
\]
\end{lem}

The specific choice of the \(\mathcal{C}^1\)-norm is immaterial here due to the compactness of \(D_1 \cup D_2\). 
Gluing constructions of this type originate in the work of Fornaess and Stout \cite{FornaessStout77a,FornaessStout77b}, and the result is now standard in holomorphic approximation theory. %An analogous result also holds for $J$-holomorphic curves in almost complex manifolds (cf.~\cite{MS12}).

\begin{proof}
Let $L$ be a short line segment connecting $D_1$ and $D_2$, and define the compact set $N := D_1 \cup L \cup D_2$ (see Figure~\ref{fig:dumbbell domain}), which is admissible (see Definition~\ref{def:admissible-set}).

By  Proposition~\ref{almost geodesic link}, we can find a continuously differentiable map \( g \colon N \to X \) satisfying \( g|_{D_i} = f_i \) for \( i = 1, 2 \), which consequently belongs to \( \mathcal{A}^1(N, X) \) by Proposition~\ref{prop:extension-admissible}.

By Theorem \ref{thm:mergelyan-manifold-valued}, there exists a holomorphic map $f$ defined on a neighborhood $\Omega$ of $N$ that approximates $g$ with $\Vert f-g\Vert_{\mathcal{C}^1}<\epsilon$ on $N$. By suitably shrinking $\Omega$, we may take it to be a dumbbell-shaped domain $S$ of the required form.
\end{proof}

\section{\bf Three Equivalent Characterizations of Weak Oka-1 Manifolds}
\label{section:2}

\begin{pro}
\label{equivalence of definitions}
The following properties of a complex manifold are equivalent:
\begin{enumerate}
\item[(i)] the weak Oka-1 property (Definition~\ref{def:xwoka1}),

\item[(ii)] the local Runge approximation property on discs (Definition~\ref{def:vwoka1}),

\item[(iii)] the Runge approximation property on discs (Definition~\ref{def:awoka1}).
\end{enumerate}
\end{pro}

\begin{proof}
{\bf (i) $\Longrightarrow$ (ii):} 
Let $f_1: \overline{\mathbb{D}} \to X$ be a holomorphic map and let $f_2:\overline{\mathbb{D}}(\frac{3}{2}, \frac{1}{3}) \to X$ be a constant map defined on the closed disc $\overline{\mathbb{D}}(\frac{3}{2}, \frac{1}{3})\subset\overline{\mathbb{D}}_2$. For any $\epsilon>0$, by Definition \ref{def:xwoka1}, there exists a holomorphic map $f: \widehat{U} \to X$ on a neighborhood $\widehat{U}$ of $\overline{\mathbb{D}}_2$ that  approximates $f_1$  with $\mathrm{d}_X(f, f_1)<
\epsilon$ on $\overline{\mathbb{D}}$.

\smallskip\noindent {\bf
(ii) $\Longrightarrow$ (iii):} 
%We first observe that every compact convex set $K \subset \mathbb{C}$ is simply connected. Consequently, by the Riemann mapping theorem, the weak Oka-1 property in Definition~\ref{def:awoka1} is equivalent to the Runge-type approximation property on discs in $\mathbb{C}$. More precisely, Definition~\ref{def:awoka1} can be restated as follows: A compact connected complex manifold $X$ with a complete distance function $\mathrm{d}_X$ has \emph{weak Oka-1 property}, if for any closed disc $\overline{\mathbb{D}}(c,r) \subset \mathbb{C}$, any holomorphic map $f \colon U \to X$ defined on a neighborhood $U$ of $\overline{\mathbb{D}}(c,r)$, and any $\epsilon > 0$, there exists a holomorphic map $\tilde{f} \colon \mathbb{C} \to X$ satisfying that $\sup\limits_{z \in \overline{\mathbb{D}}(c,r)} \mathrm{d}_X(\tilde{f}(z), f(z)) \leqslant \epsilon$.
%
%Guided by this observation, 
 Fix $\epsilon > 0$ and set $\epsilon_k = \epsilon/2^{k+1}$ for each integer $k \geqslant 0$. Beginning with a holomorphic map $f_0 \colon U \to X$ defined on a neighborhood $U$ of $\overline{\mathbb{D}}$, by Definiton \ref{def:vwoka1}, there exists a holomorphic map $f_1 \colon U_1 \to X$ defined on a neighborhood $U_1$ of $\overline{\mathbb{D}}_{2}$ such that $\sup_{z \in \overline{\mathbb{D}}} \mathrm{d}_X({f}_1(z), f_0(z)) \leqslant \epsilon_1$. 

Proceeding inductively, for each integer $k \geqslant 1$, suppose we have a holomorphic map $f_k \colon U_k \to X$ defined on a neighborhood $U_k$ of $\overline{\mathbb{D}}_{2^k}$. Applying Definition \ref{def:vwoka1} again, we obtain a holomorphic map $f_{k+1} \colon U_{k+1} \to X$ defined on a neighborhood $U_{k+1}$ of $\overline{\mathbb{D}}_{2^{k+1}}$ satisfying $\sup_{z \in \overline{\mathbb{D}}_{2^{k}}} \mathrm{d}_X({f}_{k+1}(z), f_k(z)) \leqslant \epsilon_{k+1}$. This iterative process yields a sequence of maps $\{f_k\}_{k\geqslant 0}$ whose limit exists from the error estimate. Denoted by such limit ${f}\colon \mathbb{C}\to X$. By our construction, there holds 
\[
    \sup_{z \in \overline{\mathbb{D}}} \mathrm{d}_X({f}(z), f_0(z)) \leqslant \sum\limits_{k\geqslant0}\sup\limits_{z \in \overline{\mathbb{D}}_{2^{k}}} \mathrm{d}_X({f}_{k+1}(z), f_k(z)) \leqslant \sum\limits_{k\geqslant0}\epsilon_{k+1}\leqslant\epsilon.
\]

\smallskip\noindent {\bf (iii) $\Longrightarrow$ (i):}
Let $D_1,D_2$ be disjoint closed discs in a larger closed disc $D\subset\mathbb{C}$. Let $F:U\to X$ be a holomorphic map from a neighborhood $U$ of $D_1\cup D_2$. Given an arbitrary $\epsilon > 0$, Lemma \ref{lem:dumbbell domain holomorphic map} yields a holomorphic map $G \colon S \to X$ on a dumbbell neighborhood $S$ of $D_1 \cup D_2$ such that,
\begin{equation}\label{eq:GF appro}
    \mathrm{d}_X\left(G(z),F(z)\right)\leqslant {\epsilon}/{3},\quad \text{for any $z\in {D}_1\cup {D}_2$}.
\end{equation}
%$\mathrm{d}_X\left(G(z),F(z)\right)\le \frac{e}{3}$ { for any $z\in {D}_1\cup {D}_2$}.

Since $S$ is simply connected, there exists a conformal map $\phi:S\to \mathbb{D}$. We choose a sufficiently small $\delta>0$ such that $\phi(D_1 \cup D_2) \subset \overline{\mathbb{D}}_{1-2\delta}$ and, by applying Definition \ref{def:awoka1} to $G\circ \phi^{-1} \colon \overline{\mathbb{D}}_{1-\delta}\to X$, we obtain an entire holomorphic map $H:\mathbb{C}\to X$ satisfying
\begin{equation}\label{eq:GphiH appro}
    \mathrm{d}_X\left(H(w),G\circ \phi^{-1}(w)\right)\leqslant {\epsilon}/{3},\quad \text{for any $w\in \overline{\mathbb{D}}_{1-\delta}$}.
\end{equation}

Now choose a compact set $S^\prime \subset S$ containing $D_1\cup D_2$ with path-connected complement. By Runge's approximation theorem (Theorem \ref{thm:runge approximation}) and compactness argument, there exists a polynomial $\psi:\mathbb{C}\to \mathbb{C}$ approximating $\phi$ very closely on $S^\prime$ such that $\psi\left({D}_1\cup {D}_2\right)\subset {\mathbb{D}_{1-\delta}}$ and such that
\begin{equation}\label{eq:Gphi continuous}
    \mathrm{d}_X\left({G\circ \phi^{-1}}(\psi(z)),{G\circ \phi^{-1}}(\phi(z))\right)\leqslant {\epsilon}/{3}, \quad \text{for any $z\in {D}_1\cup {D}_2$}.
\end{equation}

Define $\widehat{F}(z)\coloneq H\circ  \psi(z)\colon\mathbb{C}\to X$. Summarizing estimates (\ref{eq:GF appro})--(\ref{eq:Gphi continuous}), we have, for any $z\in{D}_1\cup {D}_2$,
\begin{align*}
    \mathrm{d}_X\left(\widehat{F}(z),{F}(z)\right)\leqslant &\, \mathrm{d}_X\left(H\circ  \psi(z),G\circ \phi^{-1}\circ \psi(z)\right)+\mathrm{d}_X\left({G\circ \phi^{-1}}(\psi(z)),{G\circ \phi^{-1}}(\phi(z))\right)\\
    &+ \mathrm{d}_X\left(G(z),F(z)\right) \leqslant\, \epsilon.
\end{align*}
Thus we conclude the proof.
\end{proof}

\section{\bf Proof of the Main Theorem}\label{section:Ahlfors Currents}

\subsection{Technical preparations}

Our proof of the Main Theorem is based on a sophisticated improvement of the ``ping-pong" strategy introduced in~\cite{Xie24}.
First of all, we need the following key result.

\begin{lem}[{\cite[Observation 3.1]{Xie24}}] \label{lem-countable-discs}
There exists a countable family of nonconstant holomorphic discs  $\left\{f_i \colon \overline{\mathbb{D}}_{\rho_i} \to X \right\}_{i\geqslant 1}$,  holomorphically defined on  larger neighborhoods, which generates all Ahlfors  currents on \( X \). 
\end{lem}

Another key ingredient in our construction  is the following lemma.

\begin{lem}\label{lem:dumbbell domain convergence}
    For each $k\geqslant 1$, let $\{{S}_{k}=\mathbb{D}\cup H_{k}\cup {D}\}_{k\geqslant1}$ be a shrinking sequence of  dumbbell domains, where $D$ is an open disk away  from $\mathbb{D}$, and where $\left\{H_k\right\}_{k\geqslant1}$ is a shrinking sequence of connecting  strips whose width decrease to $0$ as $k \to \infty$ (see Figure \ref{fig:dumbell domain convergence}). Let $\{{\phi}_{k} \colon \mathbb{D}\to  {S}_{k}\}_{k\geqslant 1}$ be a sequence of normalized conformal mappings with ${\phi}_{k}(0)=0$ and ${\phi}_{k}^\prime(0)>0$ for each $k$. Then  $\left\{{\phi}_{k}\right\}_{k\geqslant 1}$ converges to the identity map uniformly on compact set in $\mathbb{D}$ as $k \to \infty$.
\end{lem}

\begin{figure}[!htbp]
    \centering
    \includegraphics[scale=0.55]{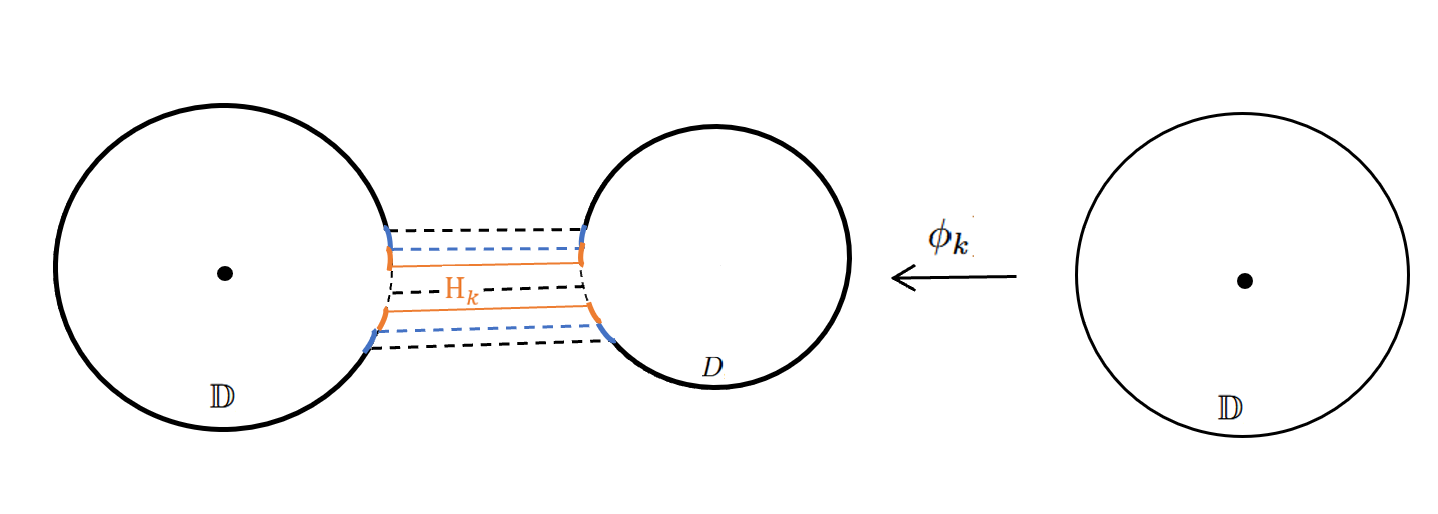}
    \caption{Dumbbell domain convergence.}
    \label{fig:dumbell domain convergence}
\end{figure}

This lemma follows directly from the Carathéodory kernel convergence theorem (cf. \cite[Theorem 1.8]{Pom92}). %a direct proof is also provided in the Appendix for completeness.
It plays a key role in our “ping-pong” strategy (see {\bf Trick 4}), a term motivated by the configuration in Figure~\ref{fig:dumbell domain convergence}, where the two discs resemble a ping-pong ball on opposite sides of a table.

\subsection{A Novel Ping-Pong Induction Process for Entire Curve Construction}

%This subsection presents a sophisticated inductive process that refines the approach in~\cite{Xie24}. 

Let $\{g_j \colon \overline{\mathbb{D}}_{r_j} \to X\}_{j \geqslant 1}$ be a sequence of nonconstant holomorphic discs that contains every $f_i$ in Lemma~\ref{lem-countable-discs} infinitely often.

Our objective is to ``amalgamate'', in a certain sense to be specified later, the family of holomorphic discs $\{g_j\}_{j \geqslant 1}$  into a sequence of concentric discs arising from one entire curve.

\subsection*{Base Case ($n = 1$)}
Fix a countable dense subset $\{\eta_j\}_{j=1}^\infty$ of $\mathcal{A}^{1,1}(X)$, the space of smooth $(1,1)$-forms on $X$ endowed with the standard Fr\'echet topology. We initialize the construction by setting:
\[
R_1 \coloneqq r_1, \quad \epsilon_1 \coloneqq 1, \quad
F_1 \coloneqq g_1 \colon \overline{\mathbb{D}}_{R_1} \to X.
\]

\subsection*{Inductive Step: From $n$ to $n+1$}
We now proceed with the inductive construction of $F_{n+1}$ from $F_n$. The input is the holomorphic disc $F_n \colon U_n \to X$, defined on a neighborhood $U_n$ of $\overline{\mathbb{D}}_{R_n}$, which was obtained in the previous step. We fix an auxiliary relatively compact disc $U'_n \Subset U_n$ containing $\overline{\mathbb{D}}_{R_n}$.

\smallskip

\subsection*{Trick 1: Set an Negligible Error Bound}

We choose a sufficiently small error bound $\epsilon_{n+1}>0$ satisfying
\begin{equation}
    \label{epsilon'n+1 condition}
     \epsilon_{n+1} < \epsilon_n / 2,
\end{equation}
such that for any holomorphic map $f \colon U_n' \to X$ approximating $F_n$ closely on $U_n'$:
\[
\sup_{z \in U'_n} \mathrm{d}_X \left( f(z), F_n(z) \right) < \epsilon_{n+1},
\]
the following two estimates hold:
\begin{equation}
    \label{epsilon' condition1}
    \left| \frac{\operatorname{Length}_\omega f(\partial \mathbb{D}_{R_n})}{\operatorname{Area}_\omega f(\mathbb{D}_{R_n})} 
    - \frac{\operatorname{Length}_\omega F_n(\partial \mathbb{D}_{R_n})}{\operatorname{Area}_\omega F_n(\mathbb{D}_{R_n})} \right| < 2^{-n},
\end{equation}
and for all $j = 1, \dots, n$,
\begin{equation}
    \label{epsilon' condition2}
    \left| \frac{\int_{\mathbb{D}_{R_n}} f^* \eta_j}{\operatorname{Area}_\omega(f(\mathbb{D}_{R_n}))} 
    - \frac{\int_{\mathbb{D}_{R_n}} F_n^* \eta_j}{\operatorname{Area}_\omega(F_n(\mathbb{D}_{R_n}))} \right| < 2^{-n}.
\end{equation}

The existence of such $\epsilon_{n+1}$ can be established by a reductio ad absurdum argument. 
Indeed, $\mathcal{C}^0$-approximation of $F_n$ by $f$ on the larger domain $U_n' \supset \overline{\mathbb{D}}_{R_n}$ 
implies $\mathcal{C}^1$-approximation on the compact subset $\overline{\mathbb{D}}_{R_n}$ via Cauchy's estimates. 
On the other hand, when $f$ is sufficiently close to $F_n$ in the $\mathcal{C}^1$-norm on $\overline{\mathbb{D}}_{R_n}$, 
the left-hand sides of  inequalities~\eqref{epsilon' condition1} and~\eqref{epsilon' condition2} will automatically be close to zero.

\smallskip
Such an $\epsilon$-trick also played a crucial role in~\cite{Xie24}. It enables us to “extend” holomorphic maps from
smaller to larger domains (up to negligible errors) by directly applying Runge’s approximation theorem,
while preserving the desired estimates.
\smallskip

\subsection*{Trick 2: Double the Disc Radius}

Set $R_{n+1}\coloneqq 2R_n$.
Applying the local Runge approximation  property of \(X\) on discs (Definition~\ref{def:vwoka1}),  we can find some open neighborhood \(V_{n+1}\) of \(\overline{\mathbb{D}}_{R_{n+1}}\) and a holomorphic map \(G_{n+1}\colon V_{n+1} \to X\) satisfying
\begin{equation}\label{eq:weak oka 1 in substep 1}
    \sup_{z \in {U}'_{n}} \mathrm{d}_X\bigl(F_n(z), G_{n+1}(z)\bigr) < \epsilon_{n+1}/8.
\end{equation}
This ``extension'' requirement is, as a matter of fact, the motivation behind Definition \ref{def:vwoka1}.

\smallskip

\subsection*{Trick 3: Bridge the Two Discs by an Almost Geodesic to a Dumbbell Domain}

To simplify notations, we normalize $g_{n+1}$ to a unit holomorphic disc by setting
\[
\widehat{g}_{n+1}(z) := g_{n+1}(r_{n+1} \cdot z) \colon \overline{\mathbb{D}} \to X.
\]
Note that this normalization, introduced purely for expository convenience, preserves the associated Ahlfors-type current, i.e., $A_{\widehat{g}_{n+1}} = A_{g_{n+1}}$ by~\eqref{eq:current def}.

Choose a point $c_{n+1} \in \mathbb{R}_+$ on the positive real axis sufficiently far from the origin such that
\[
\overline{\mathbb{D}}_{2R_n} \cap \overline{\mathbb{D}}(c_{n+1}, 1) = \emptyset.
\]
Let $L_{n+1} \subset \mathbb{R}$ be the segment on the real axis connecting $\overline{\mathbb{D}}_{R_{n+1}}=\overline{\mathbb{D}}_{2R_{n}}$ and $\overline{\mathbb{D}}(c_{n+1}, 1)$, with endpoints $q_{n+1} = R_{n+1}$ and $p_{n+1} = c_{n+1} - 1$. Define the  compact set
\[
K_{n+1} = \overline{\mathbb{D}}_{R_{n+1}}\cup L_{n+1} \cup \overline{\mathbb{D}}(c_{n+1}, 1),
\]
which is admissible by Definition~\ref{def:admissible-set}.

For every integer $k \geqslant 1$, define the map
\[
\widetilde{g}_{n+1,k}(z) := \widehat{g}_{n+1}\big((z - c_{n+1})^{k}\big) \colon \overline{\mathbb{D}}(c_{n+1}, 1) \to X.
\]
Observe that the associated Ahlfors-type currents and the length-area ratios remain unchanged:
\begin{equation}
    \label{Current is independent of k}
    A_{\widetilde{g}_{n+1,k}} = A_{g_{n+1}},
    \,\,\,
    \frac{\operatorname{Length}_\omega \widetilde{g}_{n+1, k}(\partial \mathbb{D}(c_{n+1}, 1))}{\operatorname{Area}_\omega \widetilde{g}_{n+1,k}(\mathbb{D}(c_{n+1}, 1))}
    =
    \frac{\operatorname{Length}_\omega g_{n+1}(\partial \mathbb{D}_{r_{n+1}})}{\operatorname{Area}_\omega g_{n+1}(\mathbb{D}_{r_{n+1}})},
    \quad \forall\, k \geqslant 1.
\end{equation}
This is because the identity 
\begin{equation}
    \label{naive identity 1}
 \frac{k \cdot M}{k \cdot N}=\frac{M}{N},
\quad\forall\,
M, N \in \mathbb{C},\,N \neq 0,\,k \neq 0,
\end{equation}
which directly reflects in the computations.

By Propositions~\ref{almost geodesic link} and~\ref{prop:extension-admissible} and Lemma~\ref{lem:dumbbell domain holomorphic map}, we can conclude that, 
for every $k \geqslant 1$ there exists a 
map $\chi_{n+1,k} \in \mathcal{A}^1(K_{n+1}, X)$  linking $G_{n+1}$ and $\widetilde{g}_{n+1,k}$ in the sense:
\begin{equation}\label{eq:chi=G}
    \chi_{n+1,k}\big|_{\overline{\mathbb{D}}_{R_{n+1}}} = G_{n+1}, \quad \text{and} \quad \chi_{n+1,k}\big|_{\overline{\mathbb{D}}(c_{n+1},1)} = \widetilde{g}_{n+1,k},
\end{equation}
and such that the lengths of the connecting paths $\chi_{n+1,k}(L_{n+1})$ are uniformly bounded:
\begin{equation}
    \label{key bounded lengths}
\operatorname{Length}_\omega \left( \chi_{n+1,k}(L_{n+1}) \right) \leqslant \mathrm{diam}_{\omega}(X)+1=O(1), \qquad \forall\, k \geqslant 1,
\end{equation}
where $\mathrm{diam}_{\omega}(X)$ is the diameter of $X$ with respect to  $\omega$.
  Geometrically, $K_{n+1}$ looks like a dumbbell-shaped domain with an infinitely thin neck $L_{n+1}$. Its boundary is therefore understood as
\[
\partial K_{n+1} = \partial \mathbb{D}_{R_{n+1}} + \partial\mathbb{D}(c_{n+1}, 1) + 2 L_{n+1},
\]
where the segment $L_{n+1}$ is counted twice for good reasons that will become apparent shortly (see~\eqref{eq:length-area-estimate of K_n-2}).

\smallskip
We choose a sufficiently large integer $k \gg 1$ such that the following two estimates hold:
\begin{equation}\label{eq:length-area-estimate of K_n}
    \left| \frac{\operatorname{Length}_\omega \chi_{n+1,k}(\partial K_{n+1})}{\operatorname{Area}_\omega \chi_{n+1,k}(K_{n+1})} 
    - \frac{\operatorname{Length}_\omega g_{n+1}(\partial \mathbb{D}_{r_{n+1}})}{\operatorname{Area}_\omega g_{n+1}(\mathbb{D}_{r_{n+1}})} \right| < 2^{-(n+1)},
\end{equation}
\begin{equation}\label{eq:limit curve convergence2}
    \left| \frac{\int_{K_{n+1}} \chi_{n+1,k}^* \eta_j}{\operatorname{Area}_\omega(\chi_{n+1,k}(K_{n+1}))} 
    - \frac{\int_{\mathbb{D}_{r_{n+1}}} g_{n+1}^* \eta_j}{\operatorname{Area}_\omega(g_{n+1}(\mathbb{D}_{r_{n+1}}))} \right| < 2^{-(n+1)},\quad j = 1, 2, \dots, n+1.
\end{equation}
Such a choice of \(k\) is possible because, as \(k \) tends to infinity, the left-hand sides of both the above inequalities tend to zero. This follows from~\eqref{Current is independent of k}, \eqref{key bounded lengths}, and the asymptotic identity
\begin{equation}
    \label{naive identity 2}
\lim_{k\to \infty} \frac{k \cdot M + O(1)}{k \cdot N + O(1)} = \frac{M}{N},
\quad \forall\, M, N \in \mathbb{C},\ N \neq 0,
\end{equation}
where each occurrence of \(O(1)\) denotes a complex number with  modulus  bounded from above by a positive number independent of $k$, though these bounded numbers \(O(1)\) may differ in different instances.

%The idea of considering such $K_{n+1}$ and $\chi_{n+1, k}$ originates from the work of Forn{\ae}ss and Stout~\cite{FornaessStout77a, FornaessStout77b}, who introduced this construction to facilitate the following type of Mergelyan approximation.

 \smallskip
 
\subsection*{Trick 4: Round the Dumbbell via Riemann Mapping}

Since $\chi_{n+1, k} \in \mathcal{A}^1(K_{n+1}, X)$, by a gluing technique of Fornaess and Stout~\cite{FornaessStout77a, FornaessStout77b},
 or by applying Theorem~\ref{thm:mergelyan-manifold-valued}, we can find a holomorphic map $h_{n+1, k}$, defined on a neighborhood $S_{n+1, k}$ of $K_{n+1}$, that approximates $\chi_{n+1, k}$ very closely in the $\mathcal{C}^1$-norm on $K_{n+1}$, such that in particular
\begin{equation}
    \label{h approximates chi}
    \sup_{z \in U'_{n}} \mathrm{d}_X \left(h_{n+1, k}(z),  \chi_{n+1, k}(z) \right)
    =
     \sup_{z \in U'_{n}} \mathrm{d}_X \left(h_{n+1, k}(z),  G_{n+1}(z) \right)
    < \epsilon_{n+1}/8,
\end{equation}
and such that the estimates~\eqref{eq:length-area-estimate of K_n} and~\eqref{eq:limit curve convergence2} remain valid when $\chi_{n+1, k}$ is replaced by $h_{n+1, k}$:
\begin{equation}\label{eq:length-area-estimate of K_n-new}
    \left| \frac{\operatorname{Length}_\omega h_{n+1,k}(\partial K_{n+1})}{\operatorname{Area}_\omega h_{n+1,k}(K_{n+1})} 
    - \frac{\operatorname{Length}_\omega g_{n+1}(\partial \mathbb{D}_{r_{n+1}})}{\operatorname{Area}_\omega g_{n+1}(\mathbb{D}_{r_{n+1}})} \right| < 2^{-(n+1)},
\end{equation}
\begin{equation}\label{eq:limit curve convergence2-new}
    \left| \frac{\int_{K_{n+1}} h_{n+1,k}^* \eta_j}{\operatorname{Area}_\omega(h_{n+1,k}(K_{n+1}))} 
    - \frac{\int_{\mathbb{D}_{r_{n+1}}} g_{n+1}^* \eta_j}{\operatorname{Area}_\omega(g_{n+1}(\mathbb{D}_{r_{n+1}}))} \right| < 2^{-(n+1)},\quad j = 1, 2, \dots, n+1.
\end{equation}

We then select a sufficiently small dumbbell-shaped domain $H_{n+1,k} \Subset S_{n+1,k}$ that ``almost contains'' $K_{n+1}$, 
composed of the interior $\mathring{K}_{n+1} = \mathbb{D}_{R_{n+1}} \cup \mathbb{D}(c_{n+1}, 1)$ 
together with a very thin neck region that runs parallel to and fully contains $L_{n+1}$, 
such that the following inequalities are preserved:
\begin{equation}\label{eq:length-area-estimate of K_n-2}
    \left| \frac{\operatorname{Length}_\omega h_{n+1,k}(\partial H_{n+1, k})}{\operatorname{Area}_\omega h_{n+1,k}(H_{n+1, k})} 
    - \frac{\operatorname{Length}_\omega g_{n+1}(\partial \mathbb{D}_{r_{n+1}})}{\operatorname{Area}_\omega g_{n+1}(\mathbb{D}_{r_{n+1}})} \right| < 2^{-(n+1)},
\end{equation}
\begin{equation}\label{eq:limit curve convergence2-2}
    \left| \frac{\int_{H_{n+1,k}} h_{n+1,k}^* \eta_j}{\operatorname{Area}_\omega(h_{n+1,k}(H_{n+1, k}))} 
    - \frac{\int_{\mathbb{D}_{r_{n+1}}} g_{n+1}^* \eta_j}{\operatorname{Area}_\omega(g_{n+1}(\mathbb{D}_{r_{n+1}}))} \right| < 2^{-(n+1)}, 
    \quad j = 1, 2, \dots, n+1.
\end{equation}
Such a domain $H_{n+1,k}$ exists because, as the width of the neck tends to zero, the left-hand sides of both the above inequalities converge to the corresponding left-hand sides in~\eqref{eq:length-area-estimate of K_n-new} and~\eqref{eq:limit curve convergence2-new}, respectively.

\begin{figure}[!htbp]
    \centering
    \includegraphics[scale=0.6]{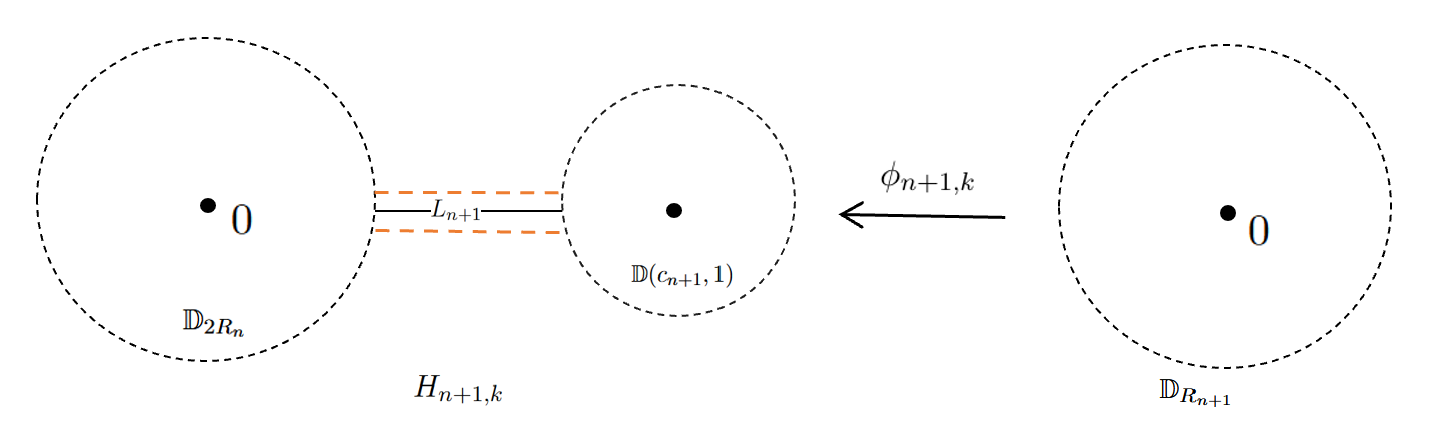}
    \caption{Conformal map from dumbbell domain to disc.}
    \label{Figure3}
\end{figure}

Furthermore, by  shrinking the neck region if necessary, Lemma~\ref{lem:dumbbell domain convergence} ensures that the holomorphic map $$\mu_{n+1} \coloneqq h_{n+1, k} \circ \phi_{n+1, k} \colon \mathbb{D}_{R_{n+1}}\rightarrow X$$ approximates  $h_{n+1, k}$ closely on $U_n'$ with:
\begin{equation}\label{G_n+1 approximates F_n}
  \sup_{z \in U'_{n}} \mathrm{d}_X \left( \mu_{n+1}(z), h_{n+1, k}(z) \right) < {\epsilon_{n+1}}/{8},
\end{equation}
where $\phi_{n+1, k} \colon \mathbb{D}_{R_{n+1}}\to H_{n+1, k} $ denotes the unique conformal mapping normalized by the conditions $\phi_{n+1, k}(0) = 0$ and $\phi_{n+1, k}'(0) > 0$; see Figure~\ref{Figure3}.

Therefore the inequalities~\eqref{eq:length-area-estimate of K_n-2} and~\eqref{eq:limit curve convergence2-2} read as \begin{equation}\label{eq:length-area-estimate-G}
    \left| \frac{\operatorname{Length}_\omega \mu_{n+1}(\partial \mathbb{D}_{R_{n+1}})}{\operatorname{Area}_\omega \mu_{n+1}(\mathbb{D}_{R_{n+1}})} 
    - \frac{\operatorname{Length}_\omega g_{n+1}(\partial \mathbb{D}_{r_{n+1}})}{\operatorname{Area}_\omega g_{n+1}(\mathbb{D}_{r_{n+1}})} \right| < 2^{-(n+1)},
\end{equation}
\begin{equation}\label{eq:cohomology-convergence-G}
    \left| \frac{\int_{\mathbb{D}_{R_{n+1}}} \mu_{n+1}^* \eta_j}{\operatorname{Area}_\omega(\mu_{n+1}(\mathbb{D}_{R_{n+1}}))} 
    - \frac{\int_{\mathbb{D}_{r_{n+1}}} g_{n+1}^* \eta_j}{\operatorname{Area}_\omega(g_{n+1}(\mathbb{D}_{r_{n+1}}))} \right| < 2^{-(n+1)}, 
    \quad j = 1, 2, \dots, n+1.
\end{equation}

 This trick, in fact, forms the cornerstone of our novel approach. It provides an appropriate adaptation of Birkhoff's classical technique~\cite{Birkhoff29} for the specific purpose of ``amalgamating" Ahlfors-type currents in our setting.
 
\subsection*{Trick 5: Slightly Enlarge the Disc}

Take a very small $0 < \delta_{n+1} \ll 1$ such that the holomorphic map
\[
F_{n+1}(z) \coloneqq \mu_{n+1}((1 - \delta_{n+1}) \cdot z),
\]
defined on a neighborhood $U_{n+1}$ of $\overline{\mathbb{D}}_{R_{n+1}}$, closely approximates $\mu_{n+1}$ on  $U_n' \Subset \mathbb{D}_{R_{n+1}}$:
\begin{equation}
    \label{F_n+1 approximates G_n+1}
\sup_{z \in U'_n} \mathrm{d}_X \left( F_{n+1}(z), \mu_{n+1}(z) \right) < \epsilon_{n+1}/8,
\end{equation}
and such that,
the estimates~\eqref{eq:length-area-estimate-G} and~\eqref{eq:cohomology-convergence-G} maintain after replacing $\mu_{n+1}$ by $F_{n+1}$:
\begin{equation}\label{eq:length-area-estimate-F}
    \left| \frac{\operatorname{Length}_\omega F_{n+1}(\partial \mathbb{D}_{R_{n+1}})}{\operatorname{Area}_\omega F_{n+1}(\mathbb{D}_{R_{n+1}})} 
    - \frac{\operatorname{Length}_\omega g_{n+1}(\partial \mathbb{D}_{r_{n+1}})}{\operatorname{Area}_\omega g_{n+1}(\mathbb{D}_{r_{n+1}})} \right| < 2^{-(n+1)},
\end{equation}
\begin{equation}\label{eq:cohomology-convergence-F}
    \left| \frac{\int_{\mathbb{D}_{R_{n+1}}} F_{n+1}^* \eta_j}{\operatorname{Area}_\omega(F_{n+1}(\mathbb{D}_{R_{n+1}}))} 
    - \frac{\int_{\mathbb{D}_{r_{n+1}}} g_{n+1}^* \eta_j}{\operatorname{Area}_\omega(g_{n+1}(\mathbb{D}_{r_{n+1}}))} \right| < 2^{-(n+1)}, 
    \quad j = 1, 2, \dots, n+1.
\end{equation}
The argument is based on continuity: as $\delta_{n+1}$ tends to zero, the left-hand sides of the two inequalities above converge to their counterparts in~\eqref{eq:length-area-estimate-G} and~\eqref{eq:cohomology-convergence-G}, respectively. Notably, the purpose of this scaling trick is precisely to ensure that $F_{n+1}$ is defined on a slightly larger domain than $\overline{\mathbb{D}}_{R_{n+1}}$, while preserving the required estimates.

This completes the ``amalgamation'' of $g_{n+1}$ into $F_{n+1}$.

\smallskip
Combining~\eqref{eq:weak oka 1 in substep 1}, \eqref{h approximates chi}, \eqref{G_n+1 approximates F_n} and~\eqref{F_n+1 approximates G_n+1}, we conclude that
\begin{equation}\label{F_n+1 approximates F_n}
\sup_{z \in U'_n} \mathrm{d}_X \left( F_{n+1}(z), F_n(z) \right) < \epsilon_{n+1}/2.
\end{equation}
 This completes the inductive construction at {\em Step} $n+1$.

\subsection{Final Construction of the Entire Curve}

By the inductive estimates~\eqref{epsilon'n+1 condition} and~\eqref{F_n+1 approximates F_n}, our obtained sequence $\{F_n\colon \overline{\mathbb{D}}_{R_n}\rightarrow X\}_{n\geqslant 1}$ of holomorphic discs converges uniformly on compact sets to an entire curve $F\colon \mathbb{C}\rightarrow X$. Indeed, for any $n\geqslant 1$ and $z \in U'_n$, we have:
\begin{align*}
    \mathrm{d}_X(F(z),F_{n}(z)) \leqslant \sum_{m\geqslant n} \mathrm{d}_X(F_{m+1}(z),F_{m}(z)) < \sum_{m\geqslant n} \frac{\epsilon_{m+1}}{2} < \sum_{m\geqslant n} \frac{\epsilon_{n+1}}{2^{m-n+1}} = \epsilon_{n+1}.
\end{align*}

This uniform  estimate ensures that $F$ inherits the geometric estimates from the approximating sequence. Specifically, \textbf{Trick 1} ensures the following proximity conditions:
\begin{equation}
    \label{F epsilon' condition1}
    \left| \frac{\operatorname{Length}_\omega F(\partial \mathbb{D}_{R_n})}{\operatorname{Area}_\omega F(\mathbb{D}_{R_n})} 
    - \frac{\operatorname{Length}_\omega F_n(\partial \mathbb{D}_{R_n})}{\operatorname{Area}_\omega F_n(\mathbb{D}_{R_n})} \right| < 2^{-n},
\end{equation}
and for all $j = 1, \dots, n$:
\begin{equation}
    \label{F-epsilon' condition2}
    \left| \frac{\int_{\mathbb{D}_{R_n}} F^* \eta_j}{\operatorname{Area}_\omega(F(\mathbb{D}_{R_n}))} 
    - \frac{\int_{\mathbb{D}_{R_n}} F_n^* \eta_j}{\operatorname{Area}_\omega(F_n(\mathbb{D}_{R_n}))} \right| < 2^{-n}.
\end{equation}

Now, let $\mathcal{T}$ be an Ahlfors current on $X$. By the construction of the sequence $\{g_j\}_{j\geqslant 1}$ (where each $f_i$ from Lemma~\ref{lem-countable-discs} appears infinitely often), there exists a subsequence $\{g_{n_j}\}_{j\geqslant 1}$ satisfying the length-area condition~\eqref{eq:length-area-Ahlfors} that generates $\mathcal{T}$.

\medskip\noindent {\bf Claim.} $\mathcal{T}$ is also generated by the concentric holomorphic discs $\{F|_{\overline{\mathbb{D}}_{R_{n_j}}}\}_{j\geqslant 1}$.

\begin{proof}
This sequence
clearly satisfies the length-area condition~\eqref{eq:length-area-Ahlfors}  by the estimates~\eqref{eq:length-area-estimate-F} and ~\eqref{F epsilon' condition1}. 
For any $\eta_k$ in our fixed dense subset $\{\eta_k\}_{k=1}^\infty $ of $\mathcal{A}^{1,1}(X)$, the inequalities~\eqref{eq:cohomology-convergence-F} and~\eqref{F-epsilon' condition2} guarantee
\[
\lim_{j\to \infty}
\frac{\int_{\mathbb{D}_{R_{n_j}}} F^* \eta_k}{\operatorname{Area}_\omega(F(\mathbb{D}_{R_{n_j}}))} =
\lim_{j\to \infty}
\frac{\int_{\mathbb{D}_{R_{n_j}}} F_{n_j}^* \eta_k}{\operatorname{Area}_\omega(F_{n_j}(\mathbb{D}_{R_{n_j}}))}
=
\lim_{j\to \infty}
\frac{\int_{\mathbb{D}_{r_{n_j}}} g_{n_j}^* \eta_k}{\operatorname{Area}_\omega(g_{n_j}(\mathbb{D}_{r_{n_j}}))}
=
\mathcal{T}(\eta_k).
\]
Therefore the sequence $\{F|_{\overline{\mathbb{D}}_{R_{n_j}}}\}_{j\geqslant 1}$ indeed generates $\mathcal{T}$ by the denseness of $\{\eta_k\}_{k\geqslant 1}$.
\end{proof}

\section{\bf Open Problems}
\label{section:4}

To conclude, we would like to emphasize the following challenges in the study of (weak) Oka-1 manifolds.

\begin{ques}
\label{question-4.1}\rm 
Let \(X\) be a compact complex manifold with the Runge approximation property on discs. Does there exist an entire curve \(f \colon \mathbb{C} \to X\) such that the  concentric holomorphic discs \(\{f|_{\overline{\mathbb{D}}_r}\}_{r>0}\) generate all Nevanlinna currents on \(X\)?
\end{ques}

 The main difficulty stems from the fact that the definition of Nevanlinna currents (cf., e.g.~\cite{WuXie25}) involves logarithmic integration $\tfrac{\mathrm{d}t}{t}$ rather than the linear measure $\mathrm{d}t$, which makes a direct adaptation of our method impossible. If the answer is affirmative, this would establish the natural analogue of Sibony's Conjecture~\ref{conj:sibony} for Nevanlinna currents.

\begin{defi}[Oka-1 manifold~\cite{AlarconForstneric23}]\label{def:Oka1}\rm
A connected complex manifold \( X \) with a complete distance function $\mathrm{d}_X$ is \emph{Oka-1} if for any open Riemann surface \( R \), compact Runge set \( K \subset R \), discrete sequence \( \{a_i\}_{i\geqslant 1}\) in $R$, continuous map \( f \colon R \to X \) holomorphic near \( K \cup \{a_i\}_{i\geqslant 1} \), \( \epsilon > 0 \), and \( k_i \in \mathbb{N} \), there exists a holomorphic map \( F \colon R \to X \) homotopic to \( f \) such that:
\begin{enumerate}
    \item \(\sup_{p \in K} \mathrm{d}_X(F(p), f(p)) < \epsilon\),
    \item \(F\) agrees with \(f\) to order \(k_i\) at each \(a_i\).
\end{enumerate}
\end{defi}

\begin{ques}[{\cite[Question 3.4]{Xie24}}]
\label{Question 3.4}
\rm 
Does the  weak Oka-1 property imply the Oka-1 property?
\end{ques}

An affirmative answer, meaning that the Runge approximation property on discs (see Proposition~\ref{equivalence of definitions}) implies Oka-1, would provide a one-dimensional analogue of Forstneri\v{c}'s celebrated theorem~\cite{Fors06} that the Runge approximation property on convex sets implies the Oka property.

\medskip

\begin{conj}
[{\cite[Conjecture 9.1]{AlarconForstneric23}}]\rm 
Every rationally connected projective manifold is Oka-1.
\end{conj}

Supporting evidence can be found in~\cite{CW23}.  
This conjecture is particularly intriguing from a complex-analytic perspective. The rational connectedness of Fano manifolds, established by Campana \cite{Cam92} and by Koll\'ar--Miyaoka--Mori \cite{KMM92}, relies on Mori's celebrated ``bend-and-break'' method \cite{MR554387}—a technique that uses characteristic $p$ arguments.  Establishing the (weak) Oka-1 property for such manifolds would  represent significant progress toward the long-standing challenge of  constructing  rational curves on Fano manifolds via purely analytic method.

\section*{\bf Acknowledgments}
Song-Yan Xie is profoundly grateful to the late Professor Nessim Sibony for sharing his conjecture  during a correspondence in January 2021. He
wishes to thank Franc Forstnerič for his inspiring online lectures on Oka manifolds in November 2022 and for insightful discussions on Question~\ref{Question 3.4}.
He also thanks Viet-Anh Nguyen for the invitation and hospitality during his visit to Lille University in September 2025.
We thank Gaofeng Huang, Yi C. Huang, Shinan Liu, Viet-Anh
Nguyen and Guangyuan Zhang for their valuable critiques  on an earlier version of this manuscript.

%Part of this work was completed during the visit of Professor John Erik Fornæss to the Academy of Mathematics and Systems Science (AMSS), Chinese Academy of Sciences, as a Distinguished Visiting Scholar. He gratefully acknowledges the hospitality of AMSS and its excellent research environment.

\section*{\bf Funding}

Song-Yan Xie is partially supported by the National Key R\&D Program of China under Grants  No. 2021YFA1003100 and No. 2023YFA1010500, as well as by the National Natural Science Foundation of China under Grants No. 12288201 and No. 12471081.

\medskip

\end{document}